\documentclass[a4paper,11pt]{article}
\usepackage{amsfonts,amssymb,amsthm,cite,amsmath,amstext}
\usepackage{hyperref}
\usepackage{color}
\usepackage[center]{titlesec}
\title{\textbf{On strongly Gauduchon metrics of compact complex manifolds}}
\author{Jian Xiao}
\date{}
\begin{document}
\maketitle

\theoremstyle{definition}
\newtheorem*{pf}{Proof}
\newtheorem{theorem}{Theorem}[section]
\newtheorem{remark}{Remark}[section]
\newtheorem{problem}{Problem}[section]
\newtheorem{conjecture}{Conjecture}[section]
\newtheorem{lemma}{Lemma}[section]
\newtheorem{corollary}{Corollary}[section]
\newtheorem{definition}{Definition}[section]
\newtheorem{proposition}{Proposition}[section]

\begin{abstract}
In this paper, we study strongly Gauduchon metrics on compact complex manifolds. We study the cohomology cones $\mathcal{SG}$ in the de Rham cohomology groups generated by all strongly Gauduchon metrics and its direct images under proper modifications. We also study the moduli of strongly Gauduchon manifolds. We prove an existence result of strongly Gauduchon metrics on compact complex manifold which is fibred over a compact complex curve. In particular, if a compact complex manifold $\mathcal{X}$ has a topologically essential fibration over a compact complex curve, and if the generic fibres satisfy $\partial\bar\partial$-lemma, then $\mathcal{X}$ admits strongly Gauduchon metrics.
\end{abstract}
%\begin{center}
\section{Introduction}
%\end{center}
Let $X$ be an $n$-dimensional compact complex manifold with a hermitian metric
$\omega$, then Gauduchon \cite{Gau77} proved that there always exists a smooth function $u$ such that $\partial \bar \partial e^{u}\omega^{n-1}=0$. We usually call a hermitian metric $\omega$ a Gauduchon metric if $\partial \bar \partial \omega^{n-1}=0$, i.e., $\partial\omega^{n-1}$ is $ \bar \partial $-closed.
Thus \cite{Gau77} established a general existence theorem of Gauduchon metrics on compact complex manifolds. Indeed, besides the elliptic PDE method in \cite{Gau77}, the existence of Gauduchon metrics is also equivalent to any plurisubharmonic function on $X$ is constant. And this is always true by the compactness of $X$ and property of plurisubharmonic fuctions. Now let's introduce the definition of strongly Gauduchon metrics which is a stronger metric notation.\\

\begin{definition}
Let $X$ be an $n$-dimensional compact complex manifold with a hermitian metric
$\omega$, if $\partial \omega^{n-1}$ is $\bar \partial$-exact, then we call $\omega$ a strongly Gauduchon metric.
\end{definition}
For convenience, we will also sometimes denote "strongly Gauduchon" briefly by SG in the
sequel. Note that any
compact complex manifold satisfying $\partial \bar \partial$-lemma admits a SG metric, since if $\omega$ is Gauduchon, then $d(\partial\omega^{n-1})=0$, and therefore $\partial \bar \partial$-lemma implies $\partial\omega^{n-1}=\bar \partial\beta$ for some smooth form $\beta$. From this point of view, we see SG is a weak metric restriction. However, for compact complex surfaces, the SG condition is equivalent to the K$\ddot{a}$hler condition, thus also equivalent to $\partial \bar \partial$-lemma. This result is essentially contained in \cite{Lam99}, and its proof relies on the Hodge theory of compact complex surfaces, especially the signature theorem.
\begin{proposition}
(\cite{Lam99}) Compact complex surface admitting a strongly Gauduchon metric is a K$\ddot{a}$hler surface.
\end{proposition}
By its very definition, it's easy to see that the existence of SG metrics is equivalent to the existence of a $d$-closed real $(2n-2)$-form $\Omega$ with its $(n-1,n-1)$-component $\Omega^{n-1,n-1}$ strictly positive. Indeed, if $\partial\omega^{n-1}=\bar \partial\beta$, then $$\Omega:=\omega^{n-1}-\beta-\bar\beta$$
is our desired form. On the other hand, if $\Omega$ is a $d$-closed real $(2n-2)$-form with $\Omega^{n-1,n-1}>0$, then the $(n-1)$-th root $\omega$ (\cite{Mich82}, page 280) of $\Omega^{n-1,n-1}$ is a SG metric. For $d\Omega=0, \Omega=\bar\Omega$ implies $\partial\omega^{n-1}=-\bar\partial\Omega^{n,n-2}$. In the sequel, we will always call such a real $(2n-2)$-form $\Omega$ a strongly Gauduchon metric.

It's obvious that a balanced manifold, i.e., complex manifold admitting a hemitian metric $\omega$ with $d\omega^{n-1}=0$ must be a SG manifold. And we will see that these two kind of  metrics are very similar. And this is one reason why we study strongly Gauduchon metrics. Fisrtly, like the balanced case \cite{Mich82}, we can also give an intrinsic characterization of SG manifold by using Hahn-Banach theorem, see \cite{Pop09}.
\begin{proposition}
(\cite{Pop09}) Let $X$ be a compact complex manifold, then $X$ is a strongly Gauduchon manifold if and only if all $d$-exact positive $(1,1)$-currents vanish on $X$.
\end{proposition}
Strongly Gauduchon metrics were firstly introduced in \cite{Pop09}, Popovici used such metrics to study the holomorphic deformation limits of projective manifolds. Let $\pi: X\rightarrow \bigtriangleup$ be a family of compact complex manifolds over a unit disk, Popovici proved that if the fibre $X_{t}$ is projective for $t\neq 0$ and the center fibre $X_{0}$ is a SG manifold, then $X_{0}$ must be a Moishezon manifold. \cite{Pop09} also announced a proof of rigidity of SG manifolds, i.e., let $\pi: X\rightarrow \bigtriangleup$ be a family of compact complex manifolds with $X_{t}$ satisfies $\partial \bar \partial$-lemma for $t\neq 0$, then the center fibre $X_{0}$ must be a SG manifold too. Thus SG metrics are useful in the study of deformation limits of projective manifolds.\\

It's well known that balanced manifolds are stable under bimeromorphic maps \cite{AB95}. A holomorphic map $\mu: \widetilde{X}\rightarrow X$ is called a proper modification if $\mu: \widetilde{X}\backslash E\rightarrow X\backslash Y$ is a biholomorphism and $Y$ is a subvariety of $X$ with codimension$(Y)\geq 2$. Assume $\mu: \widetilde{X}\rightarrow X$ is a proper modification, then $\widetilde{X}$ is balanced if and only if $X$ is balanced. Paralleling the main result of \cite{AB95}, \cite{Pop10} proved a similar result for the SG case.
\begin{theorem}
(\cite{Pop10}) Let $\mu: \widetilde{X}\rightarrow X$ be a proper modification, then $\widetilde{X}$ is strongly Gauduchon if and only if $X$ is strongly Gauduchon.
\end{theorem}
From the definition of SG manifolds, we can define a cohomology cone $\mathcal{SG} $ generated by all SG metrics, which is analogous to the K$\ddot{a}$hler cones or balanced cones (\cite{FX12}).
%%%%%%%%%
\begin{definition}
Let $X$ be a $n$-dimensional compact strongly Gauduchon manifold, the strongly Gauduchon cone $\mathcal{SG} $ is defined as
\begin{align*}
\mathcal{SG}:=\{ [\Omega]\in H^{2n-2}(X,\mathbb{R})|\Omega\ \text{is a SG metric}\},
\end{align*}
where $H^{2n-2}(X,\mathbb{R})$ is the $(2n-2)$-th de Rham cohomology group.
\end{definition}
Since SG manifolds is stable under proper modifications, it's natural to consider the induced action of $\mu$ on strongly Gauduchon cones $\mathcal{SG}$. Intuitively, strongly Gauduchon metrics are insensitive to subvarieties with codimension greater than one by degree reasons. In this paper, we obtain the following result.
\begin{theorem}
\label{pushforward}
If $\mu: \widetilde{X}\rightarrow X$ is a proper modification between strongly Gauduchon manifolds, then the push-forward map $\mu_{*}$ induced a surjective map between strongly Gauduchon cones, that is, $\mu_{*}\mathcal{SG}(\widetilde{X})=\mathcal{SG}(X)$.
\end{theorem}
%%%%%%%%%%%%%
%Besides modifications, we shall prove that strongly Gauduchon manifolds are also stable under branched covering. Recall that a holomorphic map $\pi:\widetilde{X}\rightarrow X$ is a branched covering if $\pi$ is a finite map and $\pi$ is covering outside a subvariety. This subvariety is usually called the branch locus of $\pi$, and it's usually of codimension one.
%\begin{theorem}
%\label{branch}
%Let $\pi: \widetilde{X}\rightarrow X$ be a branched covering between two compact complex manifolds, then $\widetilde{X}$ is strongly Gauduchon if and only if $X$ is strongly Gauduchon.
%\end{theorem}

In the last part of this paper, we also study the moduli of SG manifolds. Let $\pi: \mathcal{X}\rightarrow B$ be a holomorphic deformation of compact complex manifolds. It's easy to see that the SG property is preserved by small holomorphic deformation, and the strongly Gauduchon cones $\mathcal{SG}_{t}$ are invariant under parallel transport of the Gauss-Manin connection. Then it's natural to ask whether the total space is also a SG manifold if all the fibres are. In general, this does not hold. However, similar to the balanced case (\cite{Mich82}), we can prove the following result.  In \cite{Mich82}, Michelsohn proved that if a compact complex manifold admits an essential holomorphic map with balanced fibres onto a complex curve, then itself is  a balanced manifold. Here, we weaken the "essential" condition to "topologically essential" condition.
\begin{theorem}
\label{moduli}
Let $\pi: \mathcal{X}\rightarrow C$ be a topologically essential holomorphic map from a compact complex manifold $\mathcal{X}$ onto a compact complex curve $C$. If the generic fibres of $\pi$ are strongly Gauduchon manifolds, then $\mathcal{X}$ is also a strongly Gauduchon manifold.
\end{theorem}
Roughly speaking, $\pi: \mathcal{X}\rightarrow C$ is topologically essential if the fibres are not homology to zero. Indeed, the proof of this theorem is essential due to \cite{Mich82}. Since we feel this theorem is interesting in itself and haven't find the related statement in the other literatures, for readers' convenience, we will present its proof here. Moreover, it has the following interesting corollary.
\begin{corollary}
Let $\pi: \mathcal{X}\rightarrow C$ be a topologically essential holomorphic map from a compact complex manifold $\mathcal{X}$ onto a compact curve $C$. If the generic fibres of $\pi$ satisfy $\partial\bar\partial$-lemma, then $\mathcal{X}$ is a strongly Gauduchon manifold.
\end{corollary}
We can ask whether $\mathcal{X}$ also satisfies $\partial\bar\partial$-lemma. Unfortunately, there exist counterexamples in this direction. For example, the $3$-dimensional Iwasawa manifold $I_{3}$ admits a topologically essential holomorphic map onto an elliptic curve whose fibres are all K$\ddot{a}$hler, but $I_{3}$ does't satisfy $\partial\bar\partial$-lemma.
%\begin{center}
\section{Preliminaries}
%\end{center}
In this section we briefly discuss some results of currents, which we need in this paper. For more details, the readers can consult \cite{Dem09}. Let $X$ be an $n$-dimensional compact complex manifold, we denote the space of all complex or real valued smooth $k$-forms by $\mathcal{D}_{\mathbb{C}}^{k}$ or  $\mathcal{D}_{\mathbb{R}}^{k}$. For briefly, we will ignore the subscript. $\mathcal{D}^{k}$ is a Fr$\acute{e}$chet space, and it can be decomposited by the type of forms, for example,
$$\mathcal{D}^{k}=\bigoplus_{p+q=k}\mathcal{D}^{p,q}.$$
A current $T$ of degree $2n-k$ is a continuous linear functional on $\mathcal{D}^{k}$, then $T$ is just a form of degree $2n-k$ with distribution coefficients. If we denote the space of all currents of degree $2n-k$ by $\mathcal{D}^{'2n-k}$, then $\mathcal{D}^{'2n-k}$ also has a decomposition by its bidegrees,
$$\mathcal{D}^{'k}=\bigoplus_{p+q=k}\mathcal{D}^{'p,q}.$$
The space of currents $\mathcal{D}^{'2n-k}$ is dual to the space of smooth forms $\mathcal{D}^{k}$, and by this duality, we can naturally define the exterior differential operator $d,\partial,\bar\partial$ on currents. Therefore we have two de Rham complexes $(\mathcal{D}^{\bullet},d)$ and $(\mathcal{D}^{'\bullet},d)$. It's obvious that every smooth form is a current, thus we have an embedding of complex:
$$i:(\mathcal{D}^{\bullet},d)\rightarrow  (\mathcal{D}^{'\bullet},d).$$
Indeed, these two complexes give resolutions of the same sheaf of locally
constant functions, thus the morphism $i$ induces an isomorphism of cohomology
$$i:H^{\bullet}((\mathcal{D}^{\bullet},d))\widetilde{\rightarrow} H^{\bullet}((\mathcal{D}^{'\bullet},d)).$$
Thus, the de Rham cohomology groups can be defined by currents in an equivalent way. This principle also applies to Dolbeault cohomology groups.

Let $F:X\rightarrow Y$ be a holomorphic map between two compact complex manifolds, and let $T$ be a current on $X$, then the direct image of $T$ can be well defined by the pairing $\langle,\rangle$ of currents and forms. For any smooth form $\alpha$ on $Y$, $F_{*}T$ is a current such that
$$\langle F_{*}T, \alpha\rangle=\langle T,F^{*}\alpha\rangle .$$
The pairing $\langle,\rangle$ can usually represented by integration on the manifolds. Since $F$ is holomorphic, it's easy to see that $\partial F_{*}=F_{*}\partial$ and $\bar\partial F_{*}=F_{*}\bar\partial$. However, the pull-back of current is not easy to define except for some special cases.

Now let's introduce the concept of positive currents on $X$. A current $T\in \mathcal{D}^{'k,k}$ is said to be positive, if
$$ T\wedge i\alpha_{1}\wedge\overline{\alpha_{1}}\wedge\cdot\cdot\cdot\wedge i\alpha_{n-k}\wedge\overline{\alpha_{n-k}}$$
is a positive measure for all $\alpha_{1},\cdot\cdot\cdot, \alpha_{n-k}\in \mathcal{D}^{1}$. It's easy to see that positive currents are real and of order zero. In this paper, we will deal with $d$-closed positive currents. And by Lelong, it's well known that the integration current $[V]$ defined by a $p$-dimensional subvariety $V\subset X$ is a $d$-closed positive $(n-p,n-p)$-current. The support of a current $T\in \mathcal{D}^{'p}$ is the smallest closed subset $K$ such that the restriction of $T$ to $\mathcal{D}_{c}^{2n-p}(X\backslash K)$ is zero, where $\mathcal{D}_{c}^{2n-p}(X\backslash K)$ is the space of smooth forms with compact support in $X\backslash K$.
For $d$-closed positive currents, we have the following important support theorem:\\

$\bullet$ Let $V$ be a subvariety of $X$ with global irreducible components $V_{i}$ of pure dimension $p$. Then any $d$-closed current $T\in \mathcal{D}^{'n-p,n-p}$ of order zero with support in $V$ is of the form $T=\sum_{i}c_{i}[V_{i}]$ where $c_{i}\in \mathbb{C}$. Moreover, $T$ is positive if and only if all $c_{i}\geq 0$.\\

The above theorem implies any $d$-closed positive $(p,p)$-current supporting on a subvariety with codim$> p$ must vanish on $X$. In the following sections, we will deal with $d$-closed positive $(1,1)$-currents. If $T$ is such a current, then locally $T=dd^{c}u$ for some plurisubharmonic function $u$. \\

%\begin{center}
\section{Strongly Gauduchon cones $\mathcal{SG}$}
%\end{center}
In this section, we investigate some properties of strongly Gauduchon cones. Let $X$ be a compact SG manifold, recall that its SG cone is defined as following,
\begin{align*}
\mathcal{SG}(X):=\{[\Omega]\in H^{2n-2}(X,\mathbb{R})|\Omega\ \text {is a SG metric} \}.
\end{align*}
Then $\mathcal{SG}(X)$ is an open convex cone in the de Rham cohomology group. Since SG manifolds can be characterized by $d$-exact positive $(1,1)$-currents, we find it's useful to define a cone generated by $d$-closed positive currents. We denote it by $\mathcal{E}_{d}$,
\begin{align*}
\mathcal{E}_{d}(X):=\{[T]\in H^{2}(X,\mathbb{R})|T\ \text {is a positive (1,1)-current} \}.
\end{align*}\\

We claim that $\mathcal{E}_{d}(X)$ is a closed convex cone.

Fix a SG metric $\Omega$ on $X$ and a norm on the space $H^{2}(X,\mathbb{R})$. If $\{[T_{n}] \}\in \mathcal{E}_{d}(X)$ is a sequence such that $[T_{n}]\rightarrow \alpha \in H^{2}(X,\mathbb{R})$.
Then we have
\begin{align*}
|| T_{n}||_{\text {mass}}=\int_{X}T_{n}\wedge \Omega =\int_{X}T_{n}\wedge \Omega^{n-1,n-1}\leq C,
\end{align*}
where $C=C(\alpha)$ is a uniform positive constant. To see this, we just need to note that the integral $\int_{X}T_{n}\wedge \Omega $ is independent of the choices of representations in the class $[T_{n}]$. Since for any other positive current $T'_{n}\in [T_{n}]$,
$T'_{n}-T_{n}\in$Im$d$, thus
$$ || T_{n}||_{\text {mass}}= [T_{n}]\cdot[\Omega].$$
The above equality holds because the Im$d$ part contributes nothing by an application of Stoke's theorem and $d\Omega=0$. And by weak compactness of positive currents, we know there exists a convergent subsequence ${T_{n_{j}}}$ such that ${T_{n_{j}}}\rightarrow T$, where $T\in \alpha$ is a positive current. Therefore we get $\mathcal{E}_{d}=\overline{\mathcal{E}_{d}}$.\\

By Hodge theory, we know that the cup product between $H^{2}(X,\mathbb{R})$ and $H^{2n-2}(X,\mathbb{R})$ is non-degenerate and the cup product defines a natural dual of these two linear spaces. We prove that the cup product also induces a cone duality.\\

\begin{proposition}
\label{sgdual}
Assume $X$ is a compact strongly Gauduchon manifold, then $\mathcal{E}_{d}^{\vee}=\overline{\mathcal{SG}}$.
\end{proposition}
\begin{proof}
Firstly, it's obvious $\overline{\mathcal{SG}}\subseteq \mathcal{E}_{d}^{\vee}$ by the positivity of these two cones, so we only need to prove $\mathcal{E}_{d}^{\vee} \subseteq \overline{\mathcal{SG}}$.
Instead of proving $\mathcal{E}_{d}^{\vee} \subseteq \overline{\mathcal{SG}}$ directly, we prove its equivalent form, that is, the interior of $\mathcal{E}_{d}^{\vee}$ is contained in $\mathcal{SG}$. \\

We first assume $\mathcal{E}_{d}\neq \{[0]\}$.\\

Give two norms on the cohomology groups $H^{2}(X,\mathbb{R})$ and $H^{2n-2}(X,\mathbb{R})$, which we denote it by the same symbol $||\cdot||$. If $l\in \mathcal{E}_{d}^{\vee}$ is an interior point, then for small perturbation $\delta$, we still have $l+\delta\in \mathcal{E}_{d}^{\vee}$.
Here we take $\delta=-[\Omega]$ for some SG class $[\Omega]$, and $||[\Omega]||<\varepsilon$ for some small positive $\varepsilon$. Thus for any $\alpha \in \mathcal{E}_{d}$, we have $(l+\delta)\cdot \alpha\geq 0$. By linearity, we can restrict $\alpha$ to the $||\cdot||$-unit sphere. Then by compactness of $\mathcal{E}_{d}\cap$the $||\cdot||$-unit sphere, the function
$$\alpha\mapsto -\delta\cdot \alpha=[\Omega]\cdot \alpha$$
attains its maximum $M>0$ on $\mathcal{E}_{d}\cap$the $||\cdot||$-unit sphere. Thus, for interior point $l\in \mathcal{E}_{d}^{\vee}$, we have
\begin{align*}
l\cdot\alpha\geq M>0,\ \text{for} \ \alpha\in\ \mathcal{E}_{d}\cap||\cdot||-\text{unit sphere}.
\end{align*}
We now in a position to prove the following statement: \\

($\star$) let $\Phi$ be a real $d$-closed smooth
$(2n-2)$-form, if $ [\Phi]\cdot[T]\geq 0$ for all $d$-closed positive $(1,1)$-currents $T$, and
$ [\Phi]\cdot[T]= 0$ if and only if $T = 0$, then $[\Phi]\in H^{2n-2}(X,\mathbb{R})$ is a SG class.
\\

We denote the Fr$\acute{e}$chet space of real currents of degree two by $\mathcal{D}^{'2}_{\mathbb{R}}$, and fix
a hermitian metric $\omega$ on $X$. We define the mass of a positive $(1,1)$-current
T by $||T||_{mass} = \int_{X}T\wedge \omega^{n-1}$.

Then we define the following two subsets of
$\mathcal{D}^{'2}_{\mathbb{R}}$ associated to $\Phi$:
\begin{align*}
&S_{1}=\{ T\in \mathcal{D}'_{\mathbb{R}}|dT=0 \ \text{and}\ [\Phi]\cdot[T]=0\},\\
&S_{2}=\{ T\in \mathcal{D}'_{\mathbb{R}}|T\geq0 \ \text{and}\ ||T||_{mass}=1\}.
\end{align*}
Then $ S_{1}\cap S_{2}=\emptyset$ and $S_{1}$ contains the subspace $d\mathcal{D}^{'1}_{\mathbb{R}}$, moreover, $S_{1}$ is a
closed subspace and $S_{2}$ is a compact convex subset. Thus the Hahn-Banach
theorem implies there exists a real $(2n-2)$-form
$\Omega$ such that
$\Omega|_{S_{1}}=0$ and $\Omega|_{S_{2}}>0$. Here we consider $\Omega$ as a linear functional on $\mathcal{D}^{'2}_{\mathbb{R}}$. And $\Omega|_{S_{1}}=0$ implies $d\Omega=0$, $\Omega|_{S_{2}}>0$ implies $\Omega^{n-1,n-1}>0$. In the other words, $\Omega$ is a SG metric.

Since $S_{1}$ is closely related to $\Phi$ and $\Omega$ is determined by $S_{1},S_{2}$, a poriori, the SG metric $\Omega$ is related with $\Phi$. Indeed, if $[\Omega]\neq 0$ in $H^{2n-2}(X,\mathbb{R})$, there is a positive constant $c>0$ such that $[\Phi]=c[\Omega]$. Thus, $[\Phi]$ is a SG class.\\

Case $1:$ $[\Omega]\neq 0$ in $H^{2n-2}(X,\mathbb{R})$. We consider the following quotient map:
$$\pi:\{T\in \mathcal{D}^{'2}_{\mathbb{R}}\rightarrow  H^{2}(X,\mathbb{R}) \} .$$
If we consider $[\Phi],[\Omega]$ as linear functionals on $ H^{2}(X,\mathbb{R})$, then by the definition of $S_{1}$, we have
$$ker[\Phi]=\pi(S_{1})\subseteq ker[\Omega]\subseteq  H^{2}(X,\mathbb{R}).$$
Then $[\Omega]\neq 0$ yields dim$ker[\Phi]=$dim$ker[\Omega]=b_{2}-1$, where $b_{2}$ is the second betti number of $X$. So $ker[\Phi]=ker[\Omega]$, and $[\Phi]=c[\Omega]$ for some positive constant $c$.\\

Case $2:$ $[\Omega]= 0$ in $H^{2n-2}(X,\mathbb{R})$. In this case, the compactness of $X$ implies $\mathcal{SG}(X)=H^{2n-2}(X,\mathbb{R})$. For any smooth $d$-closed $(2n-2)$-form $\Theta$, there always exists a large positive constant $C$ such that $$\Theta^{n-1,n-1}+C\Omega^{n-1,n-1}>0,$$
i.e., $\Theta+C\Omega$ is a SG metric. And $[\Theta+C\Omega]=[\Theta]$ means $[\Theta]\in \mathcal{SG}(X)$. Thus by the arbitrary of $\Theta$, we have $[\Phi]\in \mathcal{SG}(X)$.\\

Next let's deal with the case which $\mathcal{E}_{d}=\{ [0]\}$.\\

Equality $\mathcal{E}_{d}=\{ [0]\}$ means that any $d$-closed positive $(1,1)$-current is indeed $d$-exact. Thus for any $T\in \mathcal{E}_{d}$, $T=dS$ for some current $S$. And since $X$ is a SG manifold, $T=0$. So $\mathcal{E}_{d}=\{ [0]\}$ implies
$d$-closed positive $(1,1)$-currents vanish, and the statement ($\star$) automatically holds true for any smooth form in $H^{2n-2}(X,\mathbb{R})$. In particular, we can choose $[\Phi]$ as a zero class in $H^{2n-2}(X,\mathbb{R})$. Then we can also produce a SG class $[\Omega]$ and the inclusion
$$ ker[\Phi]\subseteq ker[\Omega].$$
And $ker[\Phi]=H^{2n-2}(X,\mathbb{R})$ implies $ker[\Omega]=H^{2n-2}(X,\mathbb{R})$, thus the SG class $[\Omega]=[0]$ and we reach the situation as case $2$ above. \\

In summary, we have proved the equality $\mathcal{E}_{d}^{\vee}=\overline{\mathcal{SG}}.$

\end{proof}

An immediate corollary is the strictly positivity of real smooth $d$-closed $(2n-2)$-forms which degenerate only at subvarieties with codimension greater than one.
\begin{corollary}
\label{positive}
Let $X$ be a compact $n$-dimensional strongly Gauduchon manifold, and let $\Phi$ be a real smooth $d$-closed $(2n-2)$-form with $\Phi^{n-1,n-1}$ strictly positive outside a subvariety $V$. Assume the irreducible components with highest dimension of $V$ are $\{ V_{i}\}$.
If codim$V\geq 2$, or codim$V=1$ and the intersection numbers $[\Phi]\cdot[ V_{i}]$ are all positive, then $[\Phi]$ is a strongly Gauduchon class.
\end{corollary}
\begin{proof}
For any $d$-closed positive $(1,1)$-current $T$, if $[\Phi]\cdot[T]=0$, then the support of $T$ is contained in $V$. Thus by the support theorem of positive currents, we get
$$ T=\sum_{i}c_{i}[V_{i}]$$
with $c_{i}$ are nonnegative constant. And the codimension an intersection conditions implies $c_{i}=0$, i.e., $T=0$. Then an application of proposition \ref{sgdual} guarantees $[\Phi]$ is a SG class.
\end{proof}

\begin{remark}
Indeed, we have proved a cone duality $\mathcal{E}_{dd^{c}}^{\vee}= \overline{\mathcal{B}}$ for balanced cone, where $\mathcal{B}$ is the balanced cone in the Bott-Chern cohomology generated by all balanced metrics, $\mathcal{E}_{dd^{c}}$ is a closed convex cone in the Aeppli cohomology generated by $dd^{c}$-closed positive $(1,1)$-currents. For the details, see \cite{FX12}.
\end{remark}
Intuitively, corrollary \ref{positive} suggests that SG metrics are not sensitive to singularities with codimension greater than one. Note that the center of proper modification is always a subvariety with codim$\geq 2$. Now we can prove the following theorem.
\begin{theorem}
\label{push}
Assume $\mu: \widetilde{X}\rightarrow X$ is a proper modification between two compact strongly Gauduchon manifolds, i.e., $\mu: \widetilde{X}\backslash E\rightarrow X\backslash Y$ is a biholomorphism with $E=\mu^{-1}(Y)$, where $Y$ is a subvariety of $X$ with codim$\geq 2$. Then the push-forward operator induces a surjective map $$\mu_{*}:\mathcal{SG}(\widetilde{X})\rightarrow \mathcal{SG}(X).$$
\end{theorem}
Before giving the proof of theorem \ref{push} , let's give some explanation of the push-forward operator $\mu_{*}$. Since $\mu$ is a proper map, for any form or current $\widetilde{\Psi}$ on $\widetilde{X}$, its direct image $\mu_{*}\widetilde{\Psi}$ is a well defined current on $X$. In particular, $\mu_{*}$ communicates with the operator $d$, i.e., $d\mu_{*}=\mu_{*}d$, thus it induces a well defined linear map on de Rham cohomology groups:
$$\mu_{*}:H^{\bullet}(\widetilde{X},\mathbb{R})\rightarrow  H^{\bullet}(X,\mathbb{R}).$$
Now since $\mathcal{SG}(\widetilde{X})$ is in $H^{2n-2}(\widetilde{X},\mathbb{R})$, theorem \ref{push} means that the image of $\mathcal{SG}(\widetilde{X})$ coincides $\mathcal{SG}(X)$ which is contained in $H^{2n-2}(X,\mathbb{R})$.

Now we can begin our proof.
\begin{proof}
Firstly, its easy to see $\mu_{*}\mathcal{SG}(\widetilde{X})$ is contained in $\mathcal{SG}(X)$, which is an application of proposition \ref{sgdual}. Since the
pullback of $d$-closed positive $(1,1)$-currents by holomorphic map is well
defined, for any given SG metric $\widetilde{\Omega}$ on $\widetilde{X}$ and any $d$-closed positive $(1,1)$-current $T$ on $X$, the wedge product $\mu_{*}\widetilde{\Omega}\wedge T$ is well defined. Indeed, we have
$$\int_{X}\mu_{*}\widetilde{\Omega}\wedge T=\int_{\widetilde{X}} \widetilde{\Omega}\wedge \mu^{*}T.$$
In particular, we have the corresponding identity for cup product:
$$[\mu_{*}\widetilde{\Omega}]\cdot[T]=[\widetilde{\Omega}]\cdot[\mu^{*}T]\geq 0 ,$$
and the strict positivity of $\widetilde{\Omega}^{n-1,n-1}$ implies that, if $[\mu_{*}\widetilde{\Omega}]\cdot[T]=0$ then $\mu^{*}T=0$. Thus the support of $T$ is contained in $V$, and codim$V\geq 2$ yields $T=0$. From proposition \ref{sgdual}, we get $\mu_{*}\widetilde{\Omega}\in \mathcal{SG}(X)$.\\

In order to establish the other inclusion, we follows the argument of \cite{AB93}. Indeed, our case is easier. We need the following useful lemma (\cite{Var86}, 2.6) concerning the topology of subvarieties.
\begin{lemma}
Let $X$ be a compact complex manifold and let $Y$ be a subvariety
of dimension $m$, then there exist an open neighborhood $V$ of $Y$ such
that $H^{q}(V, \mathbb{C})=0$ for $q > 2m$.
\end{lemma}
In our case, since $m<n-1$, what we need is that there exist an open neighborhood $V$ of $Y$ such
that $$H^{2n-2}(V, \mathbb{C})=0.$$
Now let's fix such an open neighborhood $V$. For any SG metric $\Omega$ on $X$ and $\widetilde{\Omega}$ on $\widetilde{X}$, since $H^{2n-2}(V, \mathbb{R})=0$, we have
$$\mu_{*}\widetilde{\Omega}|_{V}=dR $$
for some current $R$ defined on $V$. Thus $\mu_{*}\widetilde{\Omega}$ is also $d$-exact on $V\backslash Y$. Moreover, since $\mu$ is biholomorphism outside $E$, $\mu_{*}\widetilde{\Omega}$ is indeed smooth on $V\backslash Y$.
This implies there exists a smooth form $\beta$ on
 $V\backslash Y$ such that
 $$\mu_{*}\widetilde{\Omega}|_{V\backslash Y}=d\beta.$$
We get $d(R-\beta)=0$ on $V\backslash Y$, thus
$R-\beta= \gamma+d\Upsilon $
on $V\backslash Y$ for some real smooth $(2n-3)$-form $\gamma$ and current $\Upsilon$.\\

We choose an open subset $W$, such $Y\subset W\Subset V$ and a cut-off function $\rho \in C_{0}^{\infty}(V)$ such that $\rho\equiv 1$ on $W$, then
$$D:=\rho(\beta+\gamma)+d(\rho\Upsilon) $$
is a well defined current on $X\backslash Y$. However, since $\beta,\gamma$ are smooth, after applying the exterior operator $d$, we get
$$ dD=d(\rho(\beta+\gamma))$$
which is a smooth real $(2n-2)$-form on $X\backslash Y$. Then its pull back by $\mu|^{*}_{\widetilde{X}\backslash E}$, which we denote by
$$\chi=\mu|^{*}_{\widetilde{X}\backslash E}dD $$
is a smooth $(2n-2)$-form on $\widetilde{X}\backslash E$. Moreover, since $\rho\equiv 1$ on $W$, we have
$$dD=d(\rho(\beta+\gamma))=d(\beta+\gamma)=dR=\mu_{*}\widetilde{\Omega} $$
on $W\backslash Y$. So on the set $\mu^{-1}(W\backslash Y)=\mu^{-1}(W)\backslash E$, we have
$$\chi|_{\mu^{-1}(W)\backslash E}=\mu|^{*}_{\mu^{-1}(W)\backslash E}(\mu_{*}\widetilde{\Omega})= \widetilde{\Omega}.$$
This implies $\chi$ and $\widetilde{\Omega}$ coincide on $\mu^{-1}(W)\backslash E$, then we can glue $\chi$ and $\widetilde{\Omega}$ together to define a global form $\widetilde{\chi}$ as following
\begin{equation*}
\widetilde{\chi}=
\begin{cases}
\chi & \text{on}\ \widetilde{X}\backslash E\\
\widetilde{\Omega} & \text{on}\ \mu^{-1}(W).
\end{cases}
\end{equation*}
Then $\widetilde{\Omega}':= \mu^{*}\Omega+ \varepsilon \widetilde{\chi}$ is a SG metric on $\widetilde{X}$ if we take $0<\varepsilon\ll 1$. Now we claim $\widetilde{\Omega}'$ satisfies
$$ [\mu_{*}\widetilde{\Omega}']=[\Omega]$$
in $H^{2n-2}(X,\mathbb{R})$. For by the definition of $\widetilde{\chi}$, we have
\begin{equation*}
\mu_{*}\widetilde{\chi}=
\begin{cases}
\mu_{*}\chi=dD & \text{on}\ X\backslash Y\\
\mu_{*}\widetilde{\Omega}=dR & \text{on}\ W.
\end{cases}
\end{equation*}
Since $D=\beta+\gamma+d\Upsilon=R$ on $W\backslash Y$, we can glue $D$ and $R$ together and define a global current on $X$ as follows
\begin{equation*}
\Theta=
\begin{cases}
D & \text{on}\ X\backslash Y\\
R & \text{on}\ W.
\end{cases}
\end{equation*}
Then we have $\mu_{*}\widetilde{\chi}=d\Theta$, and this yields
$$ \mu_{*}\widetilde{\Omega}'=\mu_{*}(\mu^{*}\Omega)+\varepsilon \mu_{*}\widetilde{\chi}=\Omega+d\Theta.$$
Now we can conclude that $\mathcal{SG}(X)\subseteq \mu_{*}\mathcal{SG}(\widetilde{X})$ by the arbitrary of $\Omega$. Combining the previous result, we finish the proof of $\mathcal{SG}(X)= \mu_{*}\mathcal{SG}(\widetilde{X})$.

\end{proof}
%%%%%
%%%%%%%%%%%%%%%%%%%%%%%%%%%%%%%%%%%%%%%%%%%%%%%%%%%%%%
%%%%%
\begin{remark}
Let's explain theorem \ref{push} a little more. A priori, the push-forward $\mu_{*}\widetilde{\Omega}$ has analytic singularities along the center $Y$. Theorem \ref{push} means that since the singularities are of codimension greater than one, we can always eliminate such singularities by a $d$-exact current, i.e., we can find a current $S$ such that $\mu_{*}\widetilde{\Omega}+dS$ is a strongly Gauduchon metric.
\end{remark}
\begin{remark}
Indeed, theorem \ref{push} is a byproduct when we studied the balanced cone $\mathcal{B}$. Since balanced manifolds are also stable under proper modifications, we initially wanted to prove $\mu_{*}\mathcal{B}(\widetilde{X})=\mathcal{B}(X)$. However, we encounter some difficulties when we wanted to define the wedge product
$$\mu_{*}\widetilde{\Omega}\wedge T,$$
where $\widetilde{\Omega}$ is a balanced metric on $\widetilde{X}$ and $T$ is just a $dd^{c}$-closed positive current. We can deal with an easy case in which $\mu$ is a blow-down map whose center($\mu$) is of dimension zero. On the other hand, this problem can be solved by the other points of view-Demailly's transcendental holomorphic Morse inequalities, see \cite{BDPP12}. Thus, if $\mu: \widetilde{X}\rightarrow X$ is a proper modification between two compact K$\ddot{a}$hler manifolds with the Picard number $\rho(X)=h^{1,1}(X)$, then $\mu_{*}\mathcal{B}(\widetilde{X})=\mathcal{B}(X).$
We want to settle the general case in a future work.\\
\end{remark}

For a compact K$\ddot{a}$hler surface, we observe that every SG metric can be deformed to be a K$\ddot{a}$hler metric.
\begin{proposition}
\label{kahler sg}
Assume $X$ is a compact K$\ddot{a}$hler surface and $\omega$ is an arbitrary hermitian metric, then there exist a smooth function $u$ and a smooth $(1,0)$-form $\theta$ such that
$e^{u}\omega+\partial\bar \theta+ \bar\partial \theta $
is a K$\ddot{a}$hler metric.
\end{proposition}
\begin{proof}
Firstly, we can choose a $u$ such that $\partial\bar\partial e^{u}\omega=0$, i.e., $e^{u}\omega$ is a Gauduchon metric. And $\partial\bar\partial$-lemma implies we can choose a smooth $(1,0)$-form $\beta$ such that
$$d(e^{u}\omega+\partial\bar \beta+ \bar\partial \beta)=0.$$

We claim that $[e^{u}\omega+\partial\bar \beta+ \bar\partial \beta]\in H^{1,1}(X,\mathbb{R})$ is a K$\ddot{a}$hler class.

This relies on the cone duality $\mathcal{K}^{\vee}=\mathcal{N}$ in \cite{BDPP12}, where $\mathcal{K}$ is the K$\ddot{a}$hler cone of $X$ and $\mathcal{N}\subseteq H^{1,1}(X,\mathbb{R})$ is the cone generated by $d$-closed positive $(1,1)$-currents. Thus there exists a smooth $(1,0)$-form $\gamma$ such that $$e^{u}\omega+\partial \overline{(\beta+\gamma)}+ \bar\partial (\beta+\gamma)$$ is a K$\ddot{a}$hler metric.
\end{proof}
We conclude this section by giving two generalizations of the above proposition \ref{kahler sg} on compact K$\ddot{a}$hler manifolds of arbitrary dimension.

On one hand, we note that $\mathcal{K}^{\vee}=\mathcal{N}$ implies every SKT metric $\omega$, i.e., $\partial\bar\partial \omega=0$, can be deformed to a K$\ddot{a}$hler metric. We can always find a smooth $(1,0)$-form $\theta$ such that $\omega+\partial\bar \theta+ \bar\partial \theta $ is a K$\ddot{a}$hler metric. On the other hand, if $X$ is a compact K$\ddot{a}$hler manifold with the Picard number $\rho(X)=h^{1,1}(X)$, then every strongly Gauduchon metric $\omega$ can be deformed to be a balanced metric. We can always find a smooth $(n-1,n-2)$-form $\theta$ such that $\omega^{n-1}+\partial\bar \theta+ \bar\partial \theta $ is a balanced metric.

%\begin{center}
\section{Moduli of strongly Gauduchon manifolds}
%\end{center}
In this section, we first discuss the small deformation of strongly Gauduchon manifolds and the variance of strongly Gauduchon cones.

Let $\pi: \mathcal{X}\rightarrow B$ be a holomorphic deformation of compact complex manifolds with base manifold $B$, then the topology of $X_{t}$ is locally constant, and therefore so are the cohomology groups $H^{k}(X_{t},\mathbb{R})$. Indeed, these cohomology groups $H^{k}(X_{t},\mathbb{R})$ constitute a vector bundle $R^{k}\pi_{*}\mathbb{R}_{\mathcal{X}}$ over $B$, where $\mathbb{R}_{\mathcal{X}}$ is the local constant sheaf over $\mathcal{X}$. Then $R^{k}\pi_{*}\mathbb{R}_{\mathcal{X}}$ is a flat bundle equipped with the well known Gauss-Manin connection $\nabla$.\\

\begin{proposition}
Let $\pi: \mathcal{X}\rightarrow B$ be a holomorphic deformation of compact complex manifolds with base manifold $B$ , then strongly Gauduchon manifolds are stable under small deformations. Moreover, if $\mathcal{X}$ is a family of strongly Gauduchon manifolds with with dim$ X_{t}=n$, the strongly Gauduchon cones $\mathcal{SG}_{t}\subseteq H^{2n-2}(X_{t},\mathbb{R})$ are invariant under the parallel transport with respect the Gauss-Manin connection $\nabla$ of the local system $R^{2n-2}\pi_{*}\mathbb{R}_{\mathcal{X}}$.
\end{proposition}
\begin{proof}
Since $\pi: \mathcal{X}\rightarrow B$ is locally trivial, for any point $t_{0}\in B$, there exists a small neighborhood $U$ of $t_{0}$ and a $C^{\infty}$-trivialization $\Phi_{U}$:
$$ \Phi_{U}:\pi^{-1}(U)\rightarrow U\times X_{t_{0}},$$
and we can even choose such a $\Phi_{U}$ such that $\Phi_{U}(t_{0},\cdot)$ is the identity map of $X_{t_{0}}$.
We denote $\Phi_{U}(t,\cdot)$ by $\Phi_{t}$.
If $X_{t_{0}}$ is a SG manifold and $\Omega_{t_{0}}$ is a SG metric, then $\Omega_{t}:=\Phi_{t}^{*}\Omega_{t_{0}}$ satifies
\begin{align*}
d\Omega_{t}=\Phi_{t}^{*}d\Omega_{t_{0}}=0,\ \text{and}\ \Omega_{t}^{n-1,n-1}>0\ \text{for}\ t\ \text{close to}\ t_{0}
\end{align*}
since $\Phi_{t}\rightarrow \Phi_{t_{0}}$ as $t\rightarrow {t_{0}}$.
Therefore, $\Omega_{t}$ will be a SG metric on $X_{t}$.

The invariance of $\mathcal{SG}_{t}$ under the parallel transport with respect the Gauss-Manin connection $\nabla$ is almost obvious. Assume
$\gamma:[0,1]\rightarrow B $, $s\mapsto \gamma(s)$ is a smooth curve in $B$. Let $\alpha(s)\in H^{2n-2}(X_{\gamma(s)},\mathbb{R})$ be a family of real $(2n-2)$-cohomology classes along the curve $\gamma$ such that
\begin{align*}
\nabla_{\gamma'(s)}\alpha(s)=0,
\end{align*}
and the initial value $\alpha(0)=[\Omega_{0}]$ for some $[\Omega_{0}]\in \mathcal{SG}_{\gamma(0)}$. By solving the above ODE, we get an injection map $\alpha(0)\mapsto \alpha(s)$.

We claim that $\alpha(s)\in \mathcal{SG}_{\gamma(s)}$ for $s$ sufficiently small.

Since the fibres $H^{2n-2}(X_{\gamma(s)},\mathbb{R})$ are locally constant, the solution $\alpha(s)$ of the above equation can be seen in the same space $H^{2n-2}(X_{\gamma(0)},\mathbb{R})$. And by the smoothness depending on $s$ of $\alpha(s)$, for any norm $||\cdot||$ in $H^{2n-2}(X_{\gamma(0)},\mathbb{R})$, we have
$$\lim_{s\rightarrow 0}|| \alpha(s)-\alpha(0)||=0.$$
Note that the quotient map from the Fr$\acute{e}$chet space of $d$-closed smooth real $(2n-2)$-forms to $H^{2n-2}(X_{\gamma(0)},\mathbb{R})$ is surjective, thus open. Therefore, for any given $\varepsilon>0$, there exists a smooth representative $\theta_{s}\in \alpha(s)$ such that
$||\theta_{s}-\Omega_{0}||_{F}<\varepsilon$ for $s$ small enough where $||\cdot||_F$ is the Fr$\acute{e}$chet norm of differential forms. Thus $\theta_{s}$ is a strongly Gauduchon metric in $\alpha(s)$, i.e., $\alpha(s)\in \mathcal{SG}_{\gamma(s)}$.

Therefore, we get an injection map from $\mathcal{SG}_{\gamma(0)}$ to $\mathcal{SG}_{\gamma(s)}$ for $s$ small enough. And by symmetry, we conclude $\mathcal{SG}_{\gamma(s)}$ is parallel with respect to the Gauss-Manin connection $\nabla$.

\end{proof}

In the following, we will always assume the base manifold $B$ is a compact curve.

Let $\pi: \mathcal{X}\rightarrow C$ be a holomorphic map from a compact complex manifold $\mathcal{X}$ onto a compact complex curve $C$, then $\mathcal{X}$ can be sliced by $\pi'$s fibres, and each fibre (may be reducible) gives a $d$-closed positive $(1,1)$-current. If $\mathcal{X}$ is strongly Gauduchon, such currents can not be a boundary. Therefore, in order to get the inverse conclusion, we should invoke some topological restriction.
\begin{definition}
Let $\pi: \mathcal{X}\rightarrow C$ be a holomorphic map from a compact complex manifold $\mathcal{X}$ onto a compact complex curve $C$, $\pi$ is called topologically essential, if any positive linear combination of components of fibres is not a boundary.
\end{definition}
It's easy to see that if the fibres are irreducible, then $\pi$ is topologically essential if and only if $\pi^{*}\omega_{C}\neq 0$ in $H^{2}(\mathcal{X},\mathbb{R})$, where $\omega_{C}$ is the fundamental class of $C$, i.e,. $\int_{C}\omega_{C}=1$. Let $p\in C$ be a regular value of $\pi$ with fibre $F_{p}:=\pi^{-1}(p)$ and let $\delta_{p}$ be the Dirac measure centered at $p$. As $\delta_{p}$ is a $d$-closed positive current on $C$, its pull-back $\pi^{*}\delta_{p}$ is well defined. Moreover, $$\pi^{*}\delta_{p}=[F_{p}].$$
To see this, we only need to verify these two currents have the same local potential. Choose a local holomorphic coordinate $t$ around $p$ on $C$ with $t(p)=0$, then $F_{p}$ is locally the zero variety of $\pi^{*}t$. So locally we have
$$[F_{p}]=dd^{c}log|\pi^{*}t|^{2}=\pi^{*}\delta_{p}.$$
Therefore, $\pi^{*}\delta_{p}=[F_{p}]$. And since $[\pi^{*}\delta_{p}]=[\pi^{*}\omega_{C}]$, $\pi^{*}\omega_{C}\neq 0$ in $H^{2}(\mathcal{X},\mathbb{R})$ is equivalent to $F_{p}\notin Imd$. The other fibres which is not given by a regular value can always be approximated by the regular ones, thus we have the same conclusion.\\

Now we can state our theorem, see also \cite{Mich82}.
\begin{theorem}
Let $\pi: \mathcal{X}\rightarrow C$ be a topologically essential holomorphic map from a compact complex manifold $\mathcal{X}$ onto a compact complex curve $C$, if all the regular fibres are strongly Gauduchon manifolds, then $\mathcal{X}$ is also a strongly Gauduchon manifolds.
\end{theorem}
\begin{proof}
We shall prove that all the positive $(1,1)$-currents $T=dS$ vanish on $\mathcal{X}$. Fix such a $T=dS$, a regular value $p\in C$ and a hermitian metric $\omega$ on $\mathcal{X}$. Take a family of smooth positive measures $\omega_{\varepsilon}$ on $C$ in the same cohomology class as the Dirac measure $\delta_{p}$ such that
$\omega_{\varepsilon}\rightarrow  \delta_{p},$
i.e., $\omega_{\varepsilon}$ is a regularization of $\delta_{p}$. \\

We will reduce our current to regular fibres by push-forward map and the fibres' SG property. However, in order to get $(1,1)$-current on fibres, we need to push forward current of bidegree $(2,2)$. Then we set
$$\widetilde{T}_{\varepsilon}:=\pi^{*}\omega_{\varepsilon}\wedge T, \widetilde{S}_{\varepsilon}:=\pi^{*}\omega_{\varepsilon}\wedge S.$$
We have $\widetilde{T}_{\varepsilon}=d\widetilde{S}_{\varepsilon}$
and $\widetilde{T}_{\varepsilon}$ is a positive $(2,2)$-current on $\mathcal{X}$. Intuitively, as $\omega_{\varepsilon}\rightarrow  \delta_{p},$ if the limit of $\{\widetilde{T}_{\varepsilon}\}$ vanishes, we will have the vanishing of $T$ on generic fibres. Thus we need first normalise the family $\{\widetilde{T}_{\varepsilon}\}$ to have uniform finite mass.

Set
$${T}_{\varepsilon}=\frac{1}{m_{\varepsilon}}\widetilde{T}_{\varepsilon}, {S}_{\varepsilon}=\frac{1}{m_{\varepsilon}}\widetilde{S}_{\varepsilon}$$
where $m_{\varepsilon}=$max$\{1, ||\widetilde{T}_{\varepsilon}||_{mass}\}$ with $||\widetilde{T}_{\varepsilon}||_{mass}=\int_{\mathcal{X}}\widetilde{T}_{\varepsilon}\wedge \omega^{n}.$ Then $||{T}_{\varepsilon}||_{mass}\leq 1$, moreover, the supports of $\{{T}_{\varepsilon}\}$ is contained in a uniform compact subset since $\omega_{\varepsilon}\rightarrow \delta_{p}$. Therefore, by the compactness of positive currents, there exist convergent subsequences in $\{{T}_{\varepsilon}\}$. Fix a convergent subsequence which we denote it also by ${T}_{\varepsilon}$ for convenience, and assume ${T}_{\varepsilon}\rightarrow
T_{\infty}$.\\

Then $T_{\infty}$ is a $d$-closed positive $(2,2)$-current on $\mathcal{X}$, and the support of $T_{\infty}$ is contained in $F_{p}$. Moreover, by the definition of $\widetilde{T}_{\varepsilon}$, $T_{\infty}$ is tangent to $F_{p}$, i.e., in the local coordinates, there is no terms having $dt$ or $d\bar t$ in $T_{\infty}$.\\

We claim $T_{\infty}=0$. As $p\in C$ is a regular value, there exists a small neighborhood $U$ of $p$ and a $C^{\infty}$-trivialization
$$\Phi_{U}:\pi^{-1}(U)\rightarrow U\times F_{p} $$
such that $\varrho:=Pr\circ \Phi_{U}$ is the identity map when restricting to $F_{p}$, where $Pr$ is the projection from $U\times F_{p}$ to $F_{p}$. Applying the push-forward operator $\varrho_{*}$ to $T_{\varepsilon}=dS_{\varepsilon}$, we get
$$\varrho_{*}T_{\varepsilon}=d\varrho_{*}S_{\varepsilon} ,$$
where $\varrho_{*}T_{\varepsilon}$ is current of degree two on $F_{p}$. Let $\varepsilon\rightarrow 0$, and by using the properties of $\varrho_{*}$ and $T_{\infty}$, we know $\varrho_{*}T_{\infty}$ is a $d$-closed positive $(1,1)$-current on $F_{p}$. Moreover, the closedness of Im$d$ implies $T_{\varepsilon}\in$Im$d$. Therefore, $$\varrho_{*}T_{\infty}=d\theta$$
for some current $\theta$ on $F_{p}$. Now since the fibre $F_{p}$ is a SG manifold, $\varrho_{*}T_{\infty}=0$ on $F_{p}$. As $T_{\infty}$ is tangent to $F_{p}$ and its support is contained in $F_{p}$, we finally get
\begin{align*}
T_{\infty}=0\ \text{on}\ \mathcal{X}.
\end{align*}
\\

Since the convergent subsequence is arbitrary, we know the whole $T_{\varepsilon}$ converges and $\lim_{\varepsilon}T_{\varepsilon}=0$.
From this we conclude $\{||\widetilde{T}_{\varepsilon}||_{mass}\}$ are uniformly finite, otherwise, by the definition of ${T}_{\varepsilon}$, there exists subsequence $\varepsilon_{k}$ such that
$$ {T}_{\varepsilon_{k}}=\frac{\widetilde{T}_{\varepsilon_{k}}}{||\widetilde{T}_{\varepsilon_{k}}||_{mass}}.$$
Thus $\lim_{k}||{T}_{\varepsilon_{k}} ||_{mass}=1$, which is a contraction. Uniformly finite of $\{||\widetilde{T}_{\varepsilon}||_{mass}\}$ implies $\{m_{\varepsilon}\}$ is uniformly finite, therefore, we have
$$\pi^{*}\omega_{\varepsilon}\wedge T=m_{\varepsilon}{T}_{\varepsilon}\rightarrow 0. $$

Denote the set of regular value of $\pi$ by $C_{r}$, then its complement $C\backslash C_{r}$ is a discrete subset of $C$. Fix a K$\ddot{a}$hler metric $\omega$ of $C$, varying $p\in C_{r}$, we get
$$T\wedge\pi^{*}\omega=0 $$
on $\mathcal{X}_{r}:=\pi^{-1}(C_{r})$. This implies almost all the slicings of $T$ are tangent to the fibres, so we have $T=g\pi^{*}\omega$ on  $\mathcal{X}_{r}$ for some positive zero-order distribution $g$ on  $\mathcal{X}_{r}$. And $dT=dg\wedge \pi^{*}\omega=0$ yields $g$ is constant when restricted to regular fibres, thus, $g=\pi^{*}f$ for some function $f$ on $C_{r}$. \\

Since $T$ is globally defined and has finite mass , we get the mass of $f\omega$ is finite on  $C_{r}$. Along with $d(f\omega)=0$, we know the trivial extension of $f\omega$ is a well defined $d$-closed $(1,1)$-current on the whole curve $C$, which we denote it by $\mu$. The pull-back $\pi^{*}\mu$ is a positive current and $[\pi^{*}\mu]=c[\pi^{*}\delta_{p}]=c[F_{p}]$ for regular value $p\in C_{r}$ and $c=\mu(C)\geq 0$. It's obvious $T-\pi^{*}\mu$ is positive and supported on $\pi^{-1}(C\backslash C_{r})$. The support theorem then implies
$$ T-\pi^{*}\mu=\sum_{i}c_{i}[E_{i}]$$
where $c_{i}$ are nonnegative constants and $E_{i}$ are irreducible components of $\pi^{-1}(C\backslash C_{r})$.
Combining $[\pi^{*}\mu]=c[\pi^{*}\delta_{p}]=c[F_{p}]$, we get
$$T-c[F_{p}]-d\eta= \sum_{i}c_{i}[E_{i}]$$
for some current $\eta$ on $\mathcal{X}$. Recall that $T=dS$, we finally get
$$c[F_{p}]+\sum_{i}c_{i}[E_{i}]=d(\eta+S)  $$
which is a boundary, and therefore, the topologically essentiality of $\pi$ implies $c=0,c_{i}=0$.

Since $c=\mu(C)$, we get $\mu=0$, thus
$T= \pi^{*}\mu+\sum_{i}c_{i}[E_{i}]=0.$
This finishes our proof and $\mathcal{X}$ is a strongly Gauduchon manifold.

\end{proof}
As $\partial\bar\partial$-lemma implies the existence of strongly Gauduchon metrics, we have the following result about the moduli of compact complex manifolds satisfying $\partial\bar\partial$-lemma.
\begin{corollary}
Let $\pi: \mathcal{X}\rightarrow C$ be a topologically essential holomorphic map from a compact complex manifold $\mathcal{X}$ onto a compact complex curve $C$, if all the regular fibres satisfy $\partial\bar\partial$-lemma, then $\mathcal{X}$ is a strongly Gauduchon manifolds.
\end{corollary}
The above corollary is optimal in the sense that there exist compact complex manifolds $\mathcal{X}$ and a topologically essential holomorphic map $$\pi: \mathcal{X}\rightarrow C$$ onto a compact curve $C$, such all the fibres satisfy $\partial\bar\partial$-lemma, but $\mathcal{X}$ does not satisfy $\partial\bar\partial$-lemma. Indeed, the $3$-dimensional Iwasawa manifold $I_{3}$ is a desired example. The Iwasawa manifold $I_{3}$ is defined as the quotient $\Gamma\backslash H$, where
\begin{gather*}
H=\left \{\begin{pmatrix}
1 & z_{1} & z_{3}\\
0 &  1 &  z_{2}\\
0 & 0 & 1
\end{pmatrix}|z_{i}\in \mathbb{C}\right \}
\end{gather*}
is the complex Heisenberg group and $\Gamma$ is the lattice defined by taking $z_{1},z_{2}, z_{3}$ to be
Gaussian integers $\mathbb{Z}[i]$, acting by left multiplication. However, the map $\pi: I_{3}\rightarrow \mathbb{C}/ \mathbb{Z}[i]$ induced by
\begin{gather*}
\begin{pmatrix}
1 & z_{1} & z_{3}\\
0 &  1 &  z_{2}\\
0 & 0 & 1
\end{pmatrix}\mapsto z_{1}
\end{gather*}
is topologically essential whose fibres are all K$\ddot{a}$hler, thus satisfying $\partial\bar\partial$-lemma.
It's well known that $I_{3}$ does't satisfy $\partial\bar\partial$-lemma.\\
\\
\\
\textbf{Acknowledgements}
I would like to thank my supervisor Prof. Jixiang Fu for his constant encouragement and Wei Xia for many helpful discussions.

\textsc{Institute of Mathematics, Fudan University, Shanghai 200433, China} \\
\textsc{Jian Xiao} \\
\verb"Email: 10110180005@fudan.edu.cn\ xiao_math@yahoo.com"

\end{document}